\documentclass[12pt,a4paper]{article}
\usepackage[a4paper]{geometry}

\usepackage{eucal}

\usepackage[all]{xy}
\SelectTips{cm}{12}

\usepackage{amssymb, eucal, amsmath, amsthm}

\usepackage{mathptmx}
\usepackage[scaled=.90]{helvet}
\usepackage{courier}

\usepackage[pdftex]{hyperref}

\hypersetup{
  colorlinks = true,
  linkcolor = black,
  urlcolor = black,
  citecolor = black
}

\newcommand{\bbZ}{\mathbb{Z}}

\newcommand{\bfB}{\mathbf{B}}
\newcommand{\bfC}{\mathbf{C}}
\newcommand{\bfD}{\mathbf{D}}
\newcommand{\bfG}{\mathbf{G}}
\newcommand{\bfS}{\mathbf{S}}

\newcommand{\rmB}{\mathrm{B}}
\newcommand{\rmC}{\mathrm{C}}
\newcommand{\rmE}{\mathrm{E}}

\newcommand{\rmK}{\mathrm{K}}
\newcommand{\rmP}{\mathrm{P}}
\newcommand{\rmS}{\mathrm{S}}
\newcommand{\rmZ}{\mathrm{Z}}

\newcommand{\calN}{\mathcal{N}}
\newcommand{\calZ}{\mathcal{Z}}

\DeclareMathOperator{\id}{id}

\DeclareMathOperator{\Hom}{Hom}
\DeclareMathOperator{\Rep}{Rep}
\DeclareMathOperator{\Hex}{Hex}

\DeclareMathOperator{\Tot}{Tot}

\DeclareMathOperator{\Mor}{Mor}
\DeclareMathOperator{\Map}{Map}

\DeclareMathOperator{\Obj}{Obj}

\DeclareMathOperator{\Ho}{Ho}

\DeclareMathOperator{\holim}{holim}

\newtheorem{theorem}{Theorem}
\newtheorem{proposition}[theorem]{Proposition}
\newtheorem{corollary}[theorem]{Corollary}
\newtheorem{lemma}[theorem]{Lemma}

\theoremstyle{definition}
\newtheorem{definition}[theorem]{Definition}
\newtheorem{example}[theorem]{Example}

\newtheorem{remark}[theorem]{Remark}

\numberwithin{theorem}{section}
\numberwithin{equation}{section}

\title{Homotopy coherent centers versus\\ centers of homotopy categories}

\author{Markus Szymik}

\date{May 2013}

\begin{document}

\maketitle

\renewcommand{\abstractname}{}

\begin{abstract}
\noindent
Centers of categories capture the natural operations on their objects. Homotopy coherent centers are introduced here as an extension of this notion to categories with an associated homotopy theory. These centers can also be interpreted as Hochschild cohomology type invariants in contexts that are not necessarily linear or stable, and we argue that they are more appropriate to higher categorical contexts than the centers of their homotopy or derived categories. Among many other things, we present an obstruction theory for realizing elements in the centers of homotopy categories, and a Bousfield-Kan type spectral sequence that computes the homotopy groups. Nontrivial classes of examples are given as illustration throughout.

\vspace{\baselineskip}

\end{abstract}


\section*{Introduction}\label{sec:intro}

If~$\bfC$ is a small category, then one may ask for a description of all natural operations on its objects. These are the families~$\Phi=(\Phi_x\colon x\to x\,|\,x\in\Obj(\bfC))$ of endomorphisms that are natural in the objects~$x$, and they form a monoid under composition. In fact, a more conceptual description of this monoid presents it as the endomorphism monoid of the identity functor~$\bfC\to\bfC$ in the category of all such functors and natural transformations. Naturality implies immediately that this monoid is always abelian. This observation is usually attributed to Eckmann and Hilton. As a classical example, for the category of commutative rings of characteristic~$p$, where~$p$ is a prime number, this monoid is the free (abelian) monoid on one generator: Frobenius. On the other hand, if the category in question is a monoid itself--a category with one object, then we have just given a long-winded description of the center of this monoid, the subset of elements that commute with all of its elements. In general, the endomorphism monoid of the identity functor is often referred to as the \emph{center} of~$\bfC$, for example in~\cite[II, \S 2]{Bass} and~\cite[II.5, Exercise~8]{MacLane}, and we will follow this terminology. Bernstein, in~\cite[1.9]{Bernstein}, defined the center of abelian categories, but his main object of interest was the category of smooth representations of a~$p$-adic group~$G$. More recently, the center (and graded versions thereof) has been investigated in various derived contexts, for examples for the derived categories of modules over (non-commutative) algebras, derived categories of coherent sheaves in algebraic geometry, as well as stable module categories in representation theory, see for example~\cite{Lowen+vandenBergh}, \cite{Avramov+Iyengar}, \cite{Buchweitz+Flenner}, and~\cite{Krause+Ye}. In all of these cases, one has been studying the homotopy category of an underlying homotopy theory. 

In the present paper, we introduce a refined notion of center for categories~$\bfC$ that admit a homotopy theory, and call it the homotopy coherent center~$\calZ(\bfC)$. This center is defined directly within~$\bfC$ rather than on the level of the homotopy category. Briefly, its elements will determine families~$\Phi=(\Phi_x\colon x\to x\,|\,x\in\Obj(\bfC))$ of endomorphisms in~$\bfC$, but it is no longer required that these are natural in the strict sense. Instead, these elements will also come with continuously chosen homotopies~\hbox{$\Phi_yf\simeq f\Phi_x$} for all arrows~\hbox{$f\colon x\to y$} in~$\bfC$ and additional higher homotopies that, for example, show that for any other arrow~$g\colon y\to z$ the two evident homotopies~\hbox{$\Phi_zgf\simeq gf\Phi_x$} that can now be obtained from~$\Phi_f$,~$\Phi_g$, and~$\Phi_{gf}$, are also homotopic. The actual formulas will bear very close resemblance to those used in the definition of Hochschild(-Mitchell) cohomology~\cite{Mitchell}, but we will work in a non-linear and unstable context: simplicial categories.

Simplicial categories are categories that are enriched in simplicial sets. (The reader unfamiliar with simplicial categories will be able to replace them by topological categories, but even in that setting (co)simplicial methods are indispensable for our approach.) Simplicial categories are by now a well-established context in which to do homotopy theory, and there is even a homotopy theory of simplicial categories themselves, see~\cite{Bergner}. This is relevant here because we show that our homotopy coherent centers are invariant under the corresponding notion of weak equivalence for simplicial categories, Dwyer-Kan equivalence~(Theorem~\ref{thm:equivalences}). We also note that the present definition automatically extends to other contexts that have an associated homotopy theory, such as Quillen model categories~\cite{Quillen} or quasi-categories, see~\cite{Boardman+Vogt},~\cite{Joyal}, and~\cite{Lurie}, by passing to the associated simplicial categories that they define. 

We show that the homotopy coherent center of any simplicial category has a canonical~$E_2$ multiplication~(Theorem~\ref{thm:E2}). This means, in particular, that these centers are~$A_\infty$ monoids~(coherently associative, i.e. have an action of an~$A_\infty=E_1$ operad), and that the multiplication is homotopy commutative~(because it extends to an action of an~$E_2$ operad). This is the analogue of Eckmann-Hilton in the present setting, and it is proved using the methods introduced by McClure and Smith in~\cite{McClure+Smith:Deligne},~\cite{McClure+Smith:survey}, and~\cite{McClure+Smith:AJM} for their solution of the Deligne conjecture.

In Section~\ref{sec:monoids}, we discuss the examples given by simplicial monoids and, in particular, simplicial groups. In the latter case, the homotopy coherent centers are related to other notions of centers studied before in homotopical group theory, see the ICM surveys~\cite{Dwyer:ICM} and~\cite{Grodal}. For example, if~$G$ is a simplicical group, then its center consists of the fixed points under the conjugation action, and the homotopy fixed points are equivalent to the homotopy coherent center as defined here (Theorem~\ref{thm:homotopy_fixed_points}). On the other hand, the homotopy center of a~$p$-compact group is defined as the loop space of the space of self-maps of the classifying space based at the identiy, see~\cite{Dwyer+Wilkerson}, and this is also equivalent to the homotopy coherent center in the case where both of them are defined~(Corollary~\ref{cor:DWmodel}).

Back in our general context, there is also a very naive and rigid notion~$\rmZ(\bfC)$ of a center for simplicial categories~$\bfC$ based on (strict) equalities. The relation of the homotopy coherent center to this simplicial center takes the form of a morphism
\[
\rmZ(\bfC)\longrightarrow\calZ(\bfC).
\]
The question whether such a map from a strict limit to a homotopy limit is an equivalence (or at least some sort of completion) is called  a homotopy limit problem, see~\cite{Thomason:HLP} and~\cite{Carlsson:HLP}. We will discuss this question for the map displayed above in Section~\ref{sec:homotopy_limit_problem} where we also give an example that shows that the situation in the present context is more complicated.

There is a canonical map~$\calZ(\bfC)\to\rmZ(\Ho\bfC)$ from the homotopy coherent center to the center (in the ordinary sense) of its homotopy category. Because the target is discrete, this map factors through the components:
\[
\pi_0\calZ(\bfC)\longrightarrow\rmZ(\Ho\bfC).
\]
In general, this latter map will be neither surjective nor injective. We explain this difference between the homotopy coherent centers of simplicial categories~$\bfC$ and the centers of their homotopy categories~$\Ho\bfC$ in Section~\ref{sec:otss}. There, using the methods initiated in~\cite{Bousfield+Kan} and~\cite{Bousfield}, we provide means to study the failure of surjectivity by developing an obstruction theory~(Theorem~\ref{thm:ot}) for the realization of an element in the target by an element in the homotopy coherent center. As for injectivity, we show that there is a spectral sequence~(Theorem~\ref{thm:ss}) with~$E_2^{0,0}\cong\rmZ(\Ho\bfC)$ that not only targets the kernel of that map but also the higher homotopy groups of the homotopy coherent center, information that is entirely invisible from the perspective of the homotopy category. The passage to the center of the homotopy category reappears from this viewpoint as an edge homomorphism. 

Fringed Bousfield-Kan type spectral sequences in general unstable contexts may have less pleasant algebraic algebraic behavior than the spectral sequences of abelian groups that one usually meets in homological algebra. In our situation, the homotopy commutative multiplication of the coherent centers leads to some simplification. This is illustrated in Section~\ref{sec:groups}, where the details are spelled out for the simplicial category of groups, where the homotopy theory is induced by conjugation, and which is related to the theory of bands~\cite{Giraud}.

The final Section~\ref{sec:groupoids} deals with another class of examples of simplicial categories: simplicial groupoids. We show that the theory presented here is related to very classical and difficult questions in (unstable) homotopy theory such as spaces of homotopy self-equivalences. That section also includes examples of simplicial categories where the homotopy coherent centers have higher homotopy groups, and where the map~\hbox{$\pi_0\calZ(\bfC)\to\rmZ(\Ho\bfC)$} is not injective or surjective.


\section{Simplicial categories and their centers}\label{sec:simplicial}

Unless otherwise stated, {\em spaces} are simplicial sets. If~$X$ and~$Y$ are simplicial sets, then~$\Mor(X,Y)$ will denote the set of simplicial maps~$X\to Y$. These are the vertices in the mapping space~$\Map(X,Y)$. The general formula is~$\Map(X,Y)_n=\Mor(\Delta^n\times X,Y)$.

A {\em simplicial category}, for us, is a small category~$\bfC$ that is enriched in spaces, and we will write
\[
\bfC(x,y)
\]
for the space of maps from the object~$x$ to the object~$y$. The only exception is the category~$\bfS$ of simplicial sets, where we have agreed to write~$\Map(X,Y)=\bfS(X,Y)$. By~$c\bfS$ we will denote the simplicial category of cosimplicial spaces.

A simplicial category is {\em locally Kan} if the mapping spaces~$\bfC(x,y)$ are Kan complexes for all choices of objects~$x$ and~$y$. This technical condition is satisfied in most examples of interest, and can be always be arranged up to weak equivalence. For example, the category~$\bfS$ itself is not locally Kan, but the full subcategory of Kan complexes is. 

A simplicial category is a simplicial object in small categories with constant object space. If~$n\geqslant0$, then will use~$\bfC_n$ as our notation for the~(ordinary) category of~$n$-simplices in~$\bfC$. In particular, we will call~$\bfC_0$ the {\em underlying category} of~$\bfC$.

There is another way to pass from a simplicial category~$\bfC$ to an ordinary category: the {\em homotopy category}~$\Ho\bfC$. It has the same set of objects, but the set of morphisms~$x\to y$ in~$\Ho\bfC$ is the set~$\pi_0\bfC(x,y)$ of components of the mapping space~$\bfC(x,y)$.


\subsection*{Simplicial centers}

We can now introduce a strict notion of center for simplicial categories.

\begin{definition}\label{def:simplicial_center}
Let~$\bfC$ be a simplicial category. The {\em simplicial center}~$\rmZ(\bfC)$ of~$\bfC$ is the equalizer (in the category of spaces) of the two maps
\begin{equation}\label{eq:lim_map}
    \prod_x\bfC(x,x)\longrightarrow\prod_{y,z}\Map(\bfC(y,z),\bfC(y,z))
\end{equation}
that are given by sending a family~$\Phi=(\Phi_x)$ to the map~$f\mapsto f\Phi_y$ and the map~$f\mapsto \Phi_zf$, respectively, in the~$(y,z)$ component.
\end{definition}

For discrete categories~$\bfC$, one may replace the target with
\[
\prod_{f\colon y\to z}\bfC(y,z),
\]
but the alternative above is adapted to work in the simplicial context as well.

One can determine the set of~$n$-simplices of the simplicial center by direct inspection.

\begin{proposition}
For every simplicial category~$\bfC$, there are isomorphisms
\[
\rmZ(\bfC)_n\cong\rmZ(\bfC_n)
\]
that are natural in~$n$. 
\end{proposition}

In particular, the elements of the center~$\rmZ(\bfC_0)$ of the underlying category~$\bfC_0$ appear as the vertices of the simplicial center~$\rmZ(\bfC)$ of~$\bfC$.

The simplicial center of a simplicial category is a simplicial monoid: a simplicial submonoid of the product of the endomorphism monoids~$\bfC(x,x)$, by Definition~\ref{def:simplicial_center}.

\begin{corollary}\label{cor:simplicial_center_is_abelian _monoid}
For every simplicial category~$\bfC$, the simplicial center~$\rmZ(\bfC)$ is a simplicial abelian monoid. 
\end{corollary}

\begin{proof}
As has been mentioned in the introduction, by Eckmann and Hilton, the centers~$\rmZ(\bfC_n)$ of the discrete categories~$\bfC_n$ are abelian. (This is an easy exercise. One way to see it is in~\cite[XI.2.4]{Kassel}.) Now the corollary follows from the preceding proposition.
\end{proof}

The following result states that the homotopy types of simplicial centers are as easy as possible.

\begin{corollary}
For every simplicial category~$\bfC$, the simplicial center~$\rmZ(\bfC)$ is a product of Eilenberg-Mac Lane spaces. 
\end{corollary}

\begin{proof}
This follows immediately from the preceding corollary and Moore's theorem about the homotopy types of simplicial abelian monoids, see \cite{Moore}.
\end{proof}

We remark that centers in other enriched contexts have been studied before, see~\cite{Lindner}.


\section{Homotopy coherent centers}\label{sec:homotopy_coherent}

The simplicial center of a simplicial category has been defined, in the preceding Section~\ref{sec:simplicial}, as a certain limit, see also Remark~\ref{rem:simplicial_center_as_limit} below. Now we will define the homotopy coherent center of a simplicial category as the corresponding homotopy limit. We will start with a brief review of the cosimplicial replacement of the homotopy end construction in this particular case, referring to \cite{Cordier+Porter:TAMS} for more general information. See also~\cite{Weibel} for the case when the source category is discrete.

\subsection*{Cosimplicial replacements}

Let~$\bfC$ be a simplicial category. For any integer~$n\geqslant0$ we can consider the space
\[
\Pi^n\bfC=\prod_{x_0,\dots,x_n}
\Map( \bfC(x_1,x_0)\times\dots\times\bfC(x_n,x_{n-1}) , \bfC(x_n,x_0) )
\]
where the product runs over the~$(n+1)$-tuples of objects of~$\bfC$. If~$\Phi$ is a vertex in~$\Pi^n\bfC$, then it can be evaluated on~$n$-tuples~$(f_1,\dots,f_n)$ of composable arrows
\[
x_0\overset{f_1}{\longleftarrow}
x_1\overset{f_2}{\longleftarrow}
x_2\longleftarrow
\dots\longleftarrow
x_{n-1}\overset{f_n}{\longleftarrow}
x_n
\]
to give~$\Phi(f_1,\dots,f_n)\in\bfC(x_n,x_0)$. Together with the evident structure maps, this defines a cosimplicial space~$\Pi^\bullet\bfC$: The coface maps~$\Pi^{n-1}\bfC\to\Pi^n\bfC$ are given by
\[
(d^k\Phi)(f_1,\dots,f_n)=
\begin{cases}
f_1\Phi(f_2,\dots,f_n) & k=0\\
\Phi(f_1,\dots,f_kf_{k+1},\dots,f_n) & 0<k<n\\
\Phi(f_1,\dots,f_{n-1})f_n & k=n
\end{cases}
\]
for~$k=0,\dots,n$, and the codegeneracy maps~$\Pi^{n+1}\bfC\to\Pi^n\bfC$ are given by
\[
(s^k\Phi)(f_1,\dots,f_n)=\Phi(f_1,\dots,f_k,\id,f_{k+1},\dots,f_n)
\]
for~$k=0,\dots,n$.

\begin{remark}\label{rem:simplicial_center_as_limit}
For~$n=0$ and~$n=1$, we obtain the source and target of~\eqref{eq:lim_map}, and we recover the definition of the simplicial center~$\rmZ(\bfC)$ as the equalizer of the two coboundaries~$\Pi^0\bfC\to\Pi^1\bfC$. This equalizer is the limit 
\[
\lim\Pi^\bullet\bfC=c\bfS(*,\Pi^\bullet\bfC)
\]
of the cosimplicial space~$\Pi^\bullet\bfC$. 
\end{remark}

We note that the cosimplicial space~$\Pi^\bullet\bfC$ is canonically pointed by the families of composition maps (interpreted as identities in low dimensions).

\begin{lemma}\label{lem:fibrant}
If~$\bfC$ is a simplicial category that is locally Kan, then the cosimplicial space~$\Pi^\bullet\bfC$ is fibrant.
\end{lemma}

\begin{proof}
This can be checked by hand: It has to be verified that certain maps
\begin{equation}\label{eq:fibrant}
\Pi^n\bfC\longrightarrow M^{n-1}\Pi^\bullet\bfC,
\end{equation}
for all~$n\geqslant0$, are fibrations. Here, the target~$M^{n-1}\Pi^\bullet\bfC\subseteq(\Pi^{n-1}\bfC)^n$ is the matching space, the subspace defined by the~$n$-tuples with consistent codegeneracies, and the map~\eqref{eq:fibrant} is also given by codegeneracies. The implies the result, because the codegeneracies are given by (projections and) evaluations on identities. The latter, in turn, are induced by insertions of identity factors, which are injective, hence cofibrations.
\end{proof}


\subsection*{Homotopy coherent centers}

We now come to the definition that is basic to the rest of this text.

\begin{definition}
Let~$\bfC$ be a simplicial category. The {\em homotopy coherent center}~$\calZ(\bfC)$ of~$\bfC$ is defined as the homotopy limit
\[
\calZ(\bfC)=\holim\Pi^\bullet\bfC=c\bfS(\Delta^\bullet,\Pi^\bullet\bfC)=\Tot(\Pi^\bullet\bfC)
\]
of the cosimplicial space~$\Pi^\bullet\bfC$.
\end{definition}

A vertex in the homotopy coherent center~$\calZ(\bfC)$ is (adjoint to) a family~$\Phi=(\Phi^n\,|\,n\geqslant0)$ of maps
\[
\Phi^n\colon\Delta^n\times\bfC(x_1,x_0)\times\dots\times\bfC(x_n,x_{n-1})
\longrightarrow\bfC(x_n,x_0)
\]
indexed by all choices of~$n$ and~$x_0,\dots,x_n$ that are consistent with the cosimplicial structure maps. If we choose~$n=0$, then we obtain a family of morphisms~$\Phi^0_x\colon x\to x$, and if we choose~$n=1$, then we obtain homotopies~$\Delta^1\times\bfC(y,z)\to\bfC(y,z)$ between the maps~\hbox{$f\mapsto\Phi^0_zf$} and~$f\mapsto f\Phi^0_y$. In particular, these give homotopies~\hbox{$\Phi^0_zf\simeq f\Phi^0_y$} that show that the homotopy classes of the~$\Phi^0_x$ define an element in the center of the homotopy category. For~$n\geqslant2$, the~$\Phi^n$ contain higher coherence information.

It follows from Lemma~\ref{lem:fibrant} that~$\calZ(\bfC)$ is a Kan complex if the simplicial category~$\bfC$ is locally Kan. See also~\cite[Proposition~2.1]{Cordier+Porter:TAMS} for a direct proof of a more general statement.

\begin{remark}\label{rem:nat}
The center of an ordinary category is the monoid of all natural transformations from the identity to itself, and we have here presented a homotopy coherent generalization of this idea that is adapted to the context of simplicial categories. For later purposes, it will be useful to know that there are similar models for spaces of coherent natural transformations between simplicial functors~$F,G\colon\bfC\to\bfD$ as well: One builds a cosimplicial space~$\Pi^\bullet(F,G)$ with
\[
\Pi^n(F,G)=\prod_{x_0,\dots,x_n}
\Map( \bfC(x_1,x_0)\times\dots\times\bfC(x_n,x_{n-1}) , \bfD(Fx_n,Gx_0) )
\]
in dimension~$n$. Then~$\calN(F,G)=\Tot\Pi^n(F,G)$ may be considered as the space of all coherent natural transformations~$F\to G$. Taking~$F=\id_\bfC=G$, we recover the homotopy coherent center as~$\calZ(\bfC)=\calN(\id_C,\id_C)$. See~\cite{Cordier+Porter:Equivariant} and~\cite{Cordier+Porter:TAMS} for this and generalizations.
\end{remark}


\section{Multiplicative structure}\label{sec:multiplicative}

We will now address the algebraic structure on the homotopy coherent centers that reflects the `composition' of coherent natural transformations.

\subsection*{An~$A_\infty$ multiplication}

We have seen in Corollary~\ref{cor:simplicial_center_is_abelian _monoid} that the simplicial center~$\rmZ(\bfC)$ of a simplicial category~$\bfC$ is a simplicial abelian monoid. I know of no reason why the homotopy coherent center~$\calZ(\bfC)$ should have a canonical simplicial monoid structure as well. However it does come with a homotopy commutative~$A_\infty$ structure, and that is just as good for the purposes of homotopy theory. To see this, we may use the criteria for operad actions on totalizations of cosimplicial objects as presented in~\cite{McClure+Smith:Deligne},~\cite{McClure+Smith:survey}, and~\cite{McClure+Smith:AJM}. 

Accordingly, we only need to find strictly associative and unital pairings
\[
m=m_{p,q}\colon\Pi^p\bfC\times\Pi^p\bfC\longrightarrow\Pi^{p+q}\bfC
\]
that satisfy
\begin{align*}
d^km(\Phi,\Psi)&=
\begin{cases}
m(d^k\Phi,\Psi)&k\leqslant p\\
m(\Phi,d^{k-p}\Psi)&k>p\\
\end{cases}\\
m(d^{p+1}\Phi,\Psi)&=m(\Phi,d^0\Psi)\\
s^km(\Phi,\Psi)&=
\begin{cases}
m(s^k\Phi,\Psi)&k<p\\
m(\Phi,s^{k-p}\Psi)&k\geqslant p\\
\end{cases}
\end{align*}
in order to deduce that~$\calZ(\bfC)=\Tot(\Pi^\bullet\bfC)$ is an~$A_\infty$ monoid~\cite[Section~3]{McClure+Smith:AJM}. The reader may check that
\[
m_{p,q}(\Phi,\Psi)(f_1,\dots,f_p,f_{p+1},\dots,f_{p+q})
=\Phi(f_1,\dots,f_p)\Psi(f_{p+1},\dots,f_{p+q})
\]
satisfies these conditions in our case.

\subsection*{Homotopy commutativity}

As suggested by the observation of Eckmann and Hilton, we will now proceed to see that the homotopy coherent center not only has an~\hbox{$A_\infty=E_1$} structure, but in fact is automatically commutative up to homotopy in the sense that this structure can be enhanced to an~$E_2$ structure.

\begin{theorem}\label{thm:E2}
 For every simplicial monoid~$\bfC$, the homotopy coherent center~$\calZ(\bfC)$ comes with an~$E_2$-multiplication; it is a homotopy commutative~$A_\infty$ monoid.
\end{theorem}  

\begin{proof}
Again, we use the work of McClure and Smith cited above. Accordingly, we need to turn~$\Pi^\bullet\bfC$ into a (nonsymmetric) operad with an associative multiplication with unit~\cite[Section~10]{McClure+Smith:AJM}.

For the operad structure, it suffices to give the insertion maps
\[
\circ_i\colon\Pi^n\bfC\times\Pi^j\bfC\longrightarrow\Pi^{n+j-1}\bfC,
\]
and we do so by defining
\[
(\Phi\circ_i\Psi)(f_1,\dots,f_{n+j-1})=
\Phi(f_1,\dots,f_{i-1},\Psi(f_i,\dots,f_{i+j-1}),f_{i+j},\dots,f_{n+j-1}).
\]
The additional structure of an associative multiplication with unit boils down to a pair of vertices~$\mu\in\Pi^2\bfC$ and~\hbox{$\epsilon\in\Pi^0\bfC$} such that the two conditions~$\mu(\mu,\id)=\mu(\id,\mu)$ and~\hbox{$\mu(\epsilon,\id)=\id=\mu(\id,\epsilon)$} are satisfied. In our case, we can define~$\epsilon$ to be the family~$(\id_x)$ of identity morphisms~$\id_x\colon x\to x$, and~$\mu$ to be the family of composition maps~\hbox{$\bfC(x_1,x_0)\times\bfC(x_2,x_1)\to\bfC(x_2,x_0)$} that send~$(f_1,f_2)$ to~$f_1f_2$.

Again, it is straightforward to verify that these maps satisfy the required conditions. 
\end{proof}

\begin{corollary}
For every simplicial monoid~$\bfC$, the set~$\pi_0\calZ(\bfC)$ of components of the homotopy coherent center~$\calZ(\bfC)$ is a commutative monoid.
\end{corollary}

See~\cite{Kock+Toen} for another non-linear version of Deligne's conjecture in a different setting.


\section{Functoriality and equivalences}\label{sec:equivalences}

In general, a simplicial functor~\hbox{$F\colon\bfC\to\bfD$} between simplicial categories does not induce a morphisms between the centers, in neither direction. This is already clear for discrete monoids. However, centers do allow for some functoriality, and this will be discussed in the present section. The main result will be the invariance of the homotopy coherent center under weak equivalences of simplicial categories as defined by Dwyer and Kan in~\cite[2.4]{Dwyer+Kan:function_complexes_in_homotopical_algebra}. Before we do so, let us discuss a different situation that also arises sufficiently often.


\subsection*{Quotients}

There is always a morphism
\[
\calZ(\bfC)\to\rmZ(\Ho\bfC),
\]
defined by sending a coherent family~$\Phi=(\Phi^n)$ to the homotopy class~$[\Phi^0]$. (One uses the homotopies~$\Phi^1$ to show that this is well-defined.) The target of the morphism is discrete, and this map clearly factors to give the second map
\[
\pi_0\calZ(\bfC)\longrightarrow\rmZ(\Ho\bfC)
\]
displayed in the introduction.

More generally, using functorial Postnikov approximations for spaces, which are simplicial functors~$\rmP_n\colon\bfS\to\bfS$ that preserve products, one may also consider the Postnikov approximations~\hbox{$\bfC\to\rmP_n\bfC$} of the simplicial categories~$\bfC$. The homotopy category is the special case~$n=0$, i.e.~$\rmP_0\bfC=\Ho\bfC$.

We obtain maps~\hbox{$\calZ(\bfC)\to\calZ(\rmP_n\bfC)$} and more generally
\[
\calZ(\rmP_m\bfC)\longrightarrow\calZ(\rmP_n\bfC)
\]
for all~$m\geqslant n$. We will later see that~$\calZ(\rmP_n\bfC)$ is~$n$-truncated, but it will not be the~$n$-truncation of~$\calZ(\bfC)$ in general. For example, in the case~$n=0$, the purpose of the obstruction theory that will be developed in Section~\ref{sec:otss} is to explain the (potential) failure of the map~\hbox{$\pi_0\calZ(\bfC)\to\rmZ(\Ho\bfC)$} to be a bijection.


\subsection*{Equivalences}

We recall from~\cite[2.4]{Dwyer+Kan:function_complexes_in_homotopical_algebra} that a simplicial functor~$F\colon\bfC\to\bfD$ between simplicial categories is called a {\it weak equivalence} if the following two conditions are satisfied: First, the induced functor~\hbox{$\Ho F\colon\Ho\bfC\to\Ho\bfD$} between the homotopy categories has to be an equivalence of (ordinary) categories, and second, for each pair of objects~$x,y$ of~$\bfC$, the induced map~\hbox{$\bfC(x,y)\to\bfD(Fx,Fy)$} has to be a weak equivalence of spaces. Two simplicial categories are {\it weakly equivalent} if there exists a (finite) zig-zag of weak equivalences between them.

\begin{theorem}\label{thm:equivalences}
Weakly equivalent simplicial categories have weakly equivalent homotopy coherent centers.
\end{theorem}

\begin{proof}
Let~$\bfC$ and~$\bfD$ be weakly equivalent simplicial categories. It will be sufficient to deal with the case in which we have a weak equivalence~$F\colon\bfC\to\bfD$ between them, and to produce a zig-zag of weak equivalences between their homotopy coherent centers. 

As an intermediate step, we will use the space~$\calN(F,F)$ of homotopy coherent natural transformations from the functor~$F$ to itself as explained in Remark~\ref{rem:nat}. For any simplicial functor~$F\colon\bfC\to\bfD$, pre- and post-composition with~$F$ induces maps
\begin{center}
  \mbox{ 
  \xymatrix{
\calZ(\bfC)\ar[r]^-{F_*}&
\calN(F,F)&
\calZ(\bfD)\ar[l]_-{F^*}
    } 
  }
\end{center}
that are compatible with the multiplications. 

If~$F\colon\bfC\to\bfD$ is a weak equivalence, then so are~$F_*$ and~$F^*$, and the result follows.
\end{proof}


\section{Monoids}\label{sec:monoids}

To illustrate the two different notions of centers that we have defined for simplicial categories so far, simplicial centers and homotopy coherent centers, let us study these in the case where the simplicial category has just one object, so that it is a simplicial monoid~$M$.

From Section~\ref{sec:simplicial}, we recall that its simplicial center~$\rmZ(M)$ is a simplicial abelian monoid with~$n$-simplices~$\rmZ(M)_n=\rmZ(M_n)$, and its homotopy type is determined by its homotopy groups since it is a product of Eilenberg-Mac Lane spaces.

Let us now inspect the homotopy coherent center~$\calZ (M)$. The definition from Section~\ref{sec:homotopy_coherent} gives
\[
\calZ(M)=\Tot(\Pi^\bullet M).
\]
By inspection, the cosimplicial space~$\Pi^\bullet M=\Map(M^\bullet,M)$ is obtained by mapping the bar construction~$M^\bullet=\rmB_\bullet(M)$ into~$M$. The bar construction~$\rmB_\bullet(M)$ is the simplicial space with~$\rmB_n(M)=M^n$, faces are given by multiplication in~$M$, except for the first and last one, which omit the corresponding entries, and degeneracies are given by inserting the identity.

The vertices of~$\calZ(M)$ are the coherent families of maps
\[
\Delta^n\longrightarrow\Map(M^n,M),
\]
or equivalently
\[
z_n\colon M^n\longrightarrow\Map(\Delta^n,M)
\]
by adjunction. (Note the different meanings of the superscript:~$M^n$ is the~$n$-th cartesian power, whereas in~$\Delta^n$ the~$n$ indicates the dimension.) A vertex in the homotopy coherent center~$\calZ(M)$ therefore gives for~$n=0$ an element~$z=z_0$ in~$M$, and for~$n=1$ and each element~$m\in M$ a path
\[
z_1(m)\colon mz\longrightarrow zm
\]
in~$M$, and for~$n=2$ it gives for each pair of elements~$m,n$ in~$M$ a~$2$-simplex~$z_2(m,n)$ in~$M$ that gives a homotopy between the two potentially different paths~$mnz\to zmn$ which are~$z_1(m\cdot n)$ and~$(z_1(m)\cdot n)\circ(m\cdot z_1(n))$. 
\begin{center}
  \mbox{ 
    \xymatrix@C=10pt{
& mzn\ar[dr]^{z_1(m)\cdot n} & \\
mnz\ar[ur]^{m\cdot z_1(n)}\ar[rr]_{z_1(m\cdot n)} && zwn
    } 
  }
\end{center}
This should give an idea of the data encoded in the homotopy coherent center in the case of simplicial monoids.

We will now specialize a bit more in order to relate our definitions to another branch of contemporary topology: homotopical group theory.


\subsection*{Simplicial groups}

In contrast to the general monoid, the elements in a group~$G$ all have inverses, so that we have a conjugation action of~$G$ on itself. The fixed points of this action form the center~$\rmZ(G)$ of~$G$ in the traditional sense. 

This suggests that the homotopy fixed point space of the conjugation action of~$G$ on itself is another natural candidate for the notion of a homotopy coherent center. It can be defined as the space
\[
\Map_G(\rmE G,G_\mathrm{conj})
\]
of equivariant maps from a free resolution~$\rmE G\to\star$ of the universal~$G$-fixed point~$\star$ to~$G_\mathrm{conj}$. The following result shows that the homotopy coherent center in our sense agrees with this concept, when the latter is defined.

\begin{theorem}\label{thm:homotopy_fixed_points}
For every simplicial group~$G$, there is an equivalence
\[
\calZ(G)\simeq\Map_G(\rmE G,G_\mathrm{conj})
\]
of its homotopy coherent center with the homotopy fixed point space of the conjugation action of~$G$ on itself.
\end{theorem}

\begin{proof}
In the right hand side, we may choose a particular model for~$\rmE G$: the geometric realization~$|\rmE_\bullet(G)|$ of the simplicial space~$\rmE_\bullet(G)$ that has~$\rmE_n(G)=G^{n+1}$, face maps omit factors, and degeneracy maps repeat them. It has a simplicial action of~$G$ (actually from both sides). 

Now an inspection shows that there is in fact an isomorphism between the two spaces in question: On the one hand, we have 
\begin{align*}
\Tot(\Map(\rmB_\bullet(G),G))
&=c\bfS(\Delta^\bullet,\Map(\rmB_\bullet(G),G))\\
&\cong\Map(\Delta^\bullet\otimes_\Delta\rmB_\bullet(G),G)\\
&\cong\Map_G(\Delta^\bullet\otimes_\Delta\rmB_\bullet(G)\times G,G_\mathrm{conj}),
\end{align*}
by the universal property of the tensor product, and on the other hand we have
\[
\Map_G(|\rmE_\bullet(G)|,G_\mathrm{conj})
=\Map_G(\Delta^\bullet\otimes_\Delta\rmE_\bullet(G),G_\mathrm{conj})
\]
by definition of the geometric realization. There are~$G$-isomorphisms
\[
\rmB_n(G)\times G\longrightarrow\rmE_n(G),
(f_1,\dots,f_n,g)\longmapsto(f_1\dots f_ng,f_2\dots f_ng,\dots,f_ng,g)
\]
with inverses
\[
\rmE_n(G)\longrightarrow\rmB_n(G)\times G,
(h_0,\dots,h_n)\longmapsto(h_0h_1^{-1},\dots,h_{n-1}h_n^{-1},h_n),
\]
respectively. These allow to connect the remaining bits.
\end{proof}

We note that the equivalence in the theorem allows us to model the map~$\rmZ(G)\to\calZ(G)$ as the inclusion of the fixed points into the homotopy fixed points.


The following (folklore) result gives another presentation of the homotopy fixed point space of the conjugation action of~$G$ on itself, and hence also of the homotopy coherent center of~$G$.

\begin{proposition}
For every simplicial group~$G$, there is an equivalence
\[
\Map_G(\rmE G,G_\mathrm{conj})\simeq\Omega(\Map(\rmB G, \rmB G),\id)
\]
with the space of loops in the space of self-maps of~$\rmB G$ which are based at the identity.
\end{proposition}

\begin{proof}
The space~$\Map_G(\rmE G,G_\mathrm{conj})$ can be identified with the space of sections of the Borel construction bundle
\[
\rmE G\times_GG_\mathrm{conj}\longrightarrow\rmB G,
\]
and the space~$\Omega(\Map(\rmB G, \rmB G),\id)$ can be identified with the space of sections of the free loop space evaluation bundle
\[
\Lambda\rmB G\longrightarrow\rmB G.
\]
In remains to note that there is an equivalence
\[
\rmE G\times_GG_\mathrm{conj}\simeq\Lambda\rmB G
\]
of spaces over~$\rmB G$. See the original sources~\cite{Goodwillie},~\cite{Burghelea+Fiedorowicz},~\cite{Bökstedt+Hsiang+Madsen}, or~\cite{Benson} for a textbook treatment.
\end{proof}

\begin{corollary}\label{cor:DWmodel}
For every simplicial group~$G$, there is an equivalence
\[
\calZ(G)\simeq\Omega(\Map(\rmB G, \rmB G),\id)
\]
of its homotopy coherent center with the space of loops in the space of self-maps of~$\rmB G$ which are based at the identity.
\end{corollary}

The right hand side of this equivalence is used a definition for the homotopy center of~$p$-compact groups in~\cite{Dwyer+Wilkerson}. 

A still more general formula for the homotopy coherent center of a simplicial groupoid will be given later, see Proposition~\ref{prop:center_formula}.


\section{The homotopy limit problem}\label{sec:homotopy_limit_problem}

It is clear from our description of the simplicial center and the homotopy coherent center as a limit and as a homotopy limit, respectively, that there is a natural map
\begin{equation}\label{eq:homotopy_limit_map}
\rmZ(\bfC)\longrightarrow\calZ(\bfC).
\end{equation}
The question whether this map is an equivalence, or how far off it is, can be compared with the difficult homotopy limit problem, see~\cite{Thomason:HLP} and~\cite{Carlsson:HLP}. 

The two notions of center clearly agree for discrete categories.

\begin{proposition}
For every discrete category~$\bfC$, the arrow
\[
\rmZ(\bfC)\longrightarrow\calZ(\bfC)
\]
is an equivalence; the domain is discrete.
\end{proposition}

Of course, this results applies, in particular, to discrete monoids, discrete groups, and discrete groupoids. 

We are now in a position to use the different descriptions of the homotopy coherent center of a simplicial group~$G$ given in Section~\ref{sec:monoids} in order to discuss the homotopy limit map~$\rmZ(G)\to\calZ(G)$ in various other special cases. Here is anther case that can be dealt with.

\begin{proposition}\label{prop:some_sag}
Let~$A$ be a simplicial abelian group. If~$A$ is an Eilenberg-Mac Lane space or if~$A$ has the homotopy type of an abelian compact Lie group, then the arrow
\[
\rmZ(A)\longrightarrow\calZ(A)
\]
is an equivalence.
\end{proposition}

\begin{proof}
We have to show that the natural map
\[
A=\Map(\star,A)\longrightarrow\Map(\rmB A,A)
\]
is an equivalence. This follows in the first case from Thom's determination of the homotopy type of the mapping space into an Eilenberg-Mac Lane space, see~\cite{Thom}. The second case follows from the same result and the fact that~$A$ is a product of a finite group with a torus.
\end{proof}

We will revisit the first case of the preceding result in Example~\ref{ex:EMspaces} from a different perspective. 

The consequence of Proposition~\ref{prop:some_sag} does not hold for all simplicial abelian groups in general, even though they are known to be products of Eilenberg-Mac Lane spaces by Moore's theorem~\cite{Moore}.

\begin{example}
If we consider~$A\simeq\bbZ\times\rmS^1$, then we have~$\rmB A\simeq\rmS^1\times\mathbb{C}\mathrm{P}^\infty$, and the components are~$\pi_0\calZ(A)\cong\bbZ^2$ in contrast to~$\pi_0\rmZ(A)\cong\bbZ$.
\end{example}

This example shows that the homotopy limit map~\eqref{eq:homotopy_limit_map} need not be an equivalence, not even after passage to classifying spaces and on homology with finite coefficients. This may be compared to the following result. 

\begin{theorem} {\upshape (\cite[Theorem~1.4]{Dwyer+Wilkerson})}
Let~$G$ be connected compact Lie group. Then the natural map
\[
\rmB\rmZ(G)\longrightarrow\Map(\rmB G,\rmB G)_{\id}
\]
induces an isomorphism on homology with finite coefficients.
\end{theorem}

Earlier results in this direction had been obtained in~\cite[Theorem~3]{Jackowski+McClure+Oliver} in the case when~$G$ is simple, and in~\cite{Dwyer+Mislin} when~$G=\mathrm{SU}(2)$.


\section{Spectral sequences and obstructions}\label{sec:otss}

In this section we answer the question about the kernel and image of the map
\[
\pi_0\calZ(\bfC)\longrightarrow\rmZ(\Ho\bfC),
\]
by means of spectral sequences and obstruction theory, respectively. 

A ubiquitous problem when discussing realization questions is the potential emptiness of realization spaces. This is irrelevant here, since the (homotopy coherent) center is always nonempty. Moreover, its set of components has the structure of an abelian monoid. This leads to a simplification of another difficulty that one often encounters when computing homotopy groups of spaces be means of such spectral sequences: the lack of~(abelian) group structures on~$\pi_0$ and~$\pi_1$.


\subsection*{Spectral sequences}

If~$\bfC$ is a simplicial category that is locally Kan, then its homotopy coherent center~$\calZ(\bfC)~$ has been defined as the totalization of the cosimplicial space~$\Pi^\bullet\bfC$. This is pointed and fibrant, see Lemma~\ref{lem:fibrant}. Bousfield and Kan, see~\cite{Bousfield+Kan} and~\cite{Bousfield} as well as the textbook~\cite[VIII.1]{Goerss+Jardine}, have shown that there is always a spectral sequence for the totalization of such a cosimplicial space. 

\begin{theorem}\label{thm:ss}
For every simplicial category~$\bfC$ that is locally Kan, there exists a fringed spectral sequence with target~$\pi_{t-s}\calZ(\bfC)$ and
\[
E_2^{s,t}=\pi^s\pi_t\Pi^\bullet\bfC
\]
for~$t\geqslant s\geqslant0$. In particular,
\[
E_2^{0,0}\cong\rmZ(\Ho\bfC)
\] 
is the center of its homotopy category.
\end{theorem}

\begin{proof}
Only the last statement requires proof. The~$E_1$ term is given as
\[
E_1^{s,t}=N^s\pi_t\Pi^\bullet\bfC.
\]
In particular, the~$0$-line is
\[
E_1^{0,t}=\pi_t\Pi^0\bfC
\]
with differential given by the difference between the two coface operators~$d_0$ and~$d_1$. This gives
\[
E_2^{0,t}=\pi^0\pi_t\Pi^\bullet\bfC.
\]
For~$t=0$, we have
\[
E_1^{0,0}=\prod_x\pi_0\bfC(x,x),
\] 
and this consists of the families~$\Phi$ homotopy classes~$\Phi_x\colon x\to x$. These lie in the equalizer~$\pi^0\pi_0\Pi^\bullet\bfC$ if and only if~$f\Phi_x$ and~$\Phi_y f$ are homotopic for all~$f\colon x\to y$ in~$\bfC$, and this is the case if and only if~$\Phi$ is in the center~$\rmZ(\Ho\bfC)$ of the homotopy category~$\Ho\bfC$.
\end{proof}

Compared with the general spectral sequence of Bousfield-Kan type, the spectral sequences in Theorem~\ref{thm:ss} have a relatively well-behaved left lower corner. The monoidal structure on~$\Pi^\bullet\bfC$ leads to abelian group structures from the~$E_1$ term on, with the possible exception of~$E_r^{0,0}$. In that spot, our spectral sequence starts with
\[
E_1^{0,0}=\prod_x\pi_0\bfC(x,x),
\] 
which--in the most general case that we consider--is just a monoid. But, the structure~\hbox{$E_2^{0,0}\cong\rmZ(\Ho\bfC)$} on the next term is already an abelian monoid. This is clearly as good as one may hope for in the context of centers. See also the following Section~\ref{sec:groups} that contains a detailed description of these phenomena in a case that displays these difficulties (and only these).

Convergence of Bousfield-Kan type spectral sequences is often a delicate issue, and we do not have to add anything to the original results here. See again~\cite{Bousfield+Kan},~\cite[\S4]{Bousfield}, and~\cite[VI.2]{Goerss+Jardine}. We would only like to point out the following result for truncated situations.

\begin{proposition}
If~$\bfC$ is a locally Kan simplicial category such that its mapping spaces~$\bfC(x,y)$ are~$n$-truncated for an~$n$ that is independent of the objects~$x$ and~$y$, then the spectral sequence of Theorem~\ref{thm:ss} converges completely to its target. 
\end{proposition}

This proposition applies, in particular, to simplicial categories that arise form categories enriched in groupoids ($n=1$).

Before we turn towards the obstruction theory that comes with a fringed Bousfield-Kan type spectral sequence, let us briefly collect the different types of truncations for homotopy coherent centers that we have considered so far. For all~$s\geqslant0$ and~$n=t-s\geqslant0$ there are maps
\begin{center}
  \mbox{ 
    \xymatrix{
\calZ(\bfC)\ar@{=}[d]\ar[r]&\rmP_n\calZ(\bfC)\ar[rr]&&\calZ(\rmP_n\bfC)\\
\Tot(\Pi^\bullet\bfC)\ar[rr]&&\Tot_s(\Pi^\bullet\bfC)
    } 
  }
\end{center}
and from the position of the corresponding regions in our spectral sequence it is clear that one cannot expect these to be equivalences in general.


\subsection*{Obstructions}

The fringed spectral sequence of Theorem~\ref{thm:ss} comes with an obstruction theory that allows the investigation of the edge homomorphism
\[
\pi_0\calZ(\bfC)\longrightarrow\rmZ(\Ho\bfC),
\]
in particular of its image. Note that the range of that map is the~$E_2^{0,0}$ term of the spectral sequence, whereas the domain belongs to the target of that spectral sequence. Therefore, we can start our discussion with an element of~$\rmZ(\Ho\bfC)\cong\pi_0\Tot_1(\Pi^\bullet\bfC)$ and ask whether or not we can lift it to~$\pi_0\Tot_s(\Pi^\bullet\bfC)$ for a given~$s$. We will be able to do this for every~$s$ if and only if the given element lifts to~$\pi_0\Tot(\Pi^\bullet\bfC)=\pi_0\calZ(\bfC)$. 

Since this is by now standard, we will only sketch how the general obstruction theory applies to our specific context and refer to~\cite[\S5 and~\S10]{Bousfield} and~\cite[VIII.4]{Goerss+Jardine} for details. 

First of all the representatives of the classes in~$\pi_0\Tot_s(\Pi^\bullet\bfC)$ can be fairly explicitly described. They are given by~$(s+1)$-tuples~\hbox{$(\Phi_0,\dots,\Phi_s)$} of simplices
\[
\Phi_p\colon\Delta^p\longrightarrow\Pi^p\bfC
\]
that are compatible with the coface and codegeneracy maps whenever these are defined. The maps~\hbox{$(\Phi_0,\dots,\Phi_s)\mapsto(\Phi_0,\dots,\Phi_{s-1})$} and~$(\Phi_0,\dots,\Phi_s)\mapsto\Phi_s$ induce the upper left of the following two pullback squares.
\begin{center}
  \mbox{ 
    \xymatrix{
    \Tot_s(\Pi^\bullet\bfC)\ar[r]\ar[d]&\Map(\Delta^s,\Pi^s\bfC)\ar[d]^f&\\
    \Tot_{s-1}(\Pi^\bullet\bfC)\ar[r]&
    P\ar[r]\ar[d]&\Map(\Delta^s,M^{s-1}\Pi^\bullet\bfC)\ar[d]\\
    &\Map(\partial\Delta^s,\Pi^s\bfC)\ar[r]&\Map(\partial\Delta^s,M^{s-1}\Pi^\bullet\bfC)
    } 
  }
\end{center}
The map~$f$ is the fibration determined by the fibrant cosimplicial space~$\Pi^\bullet\bfC$. Thus, in order to produce a lift of an element in~$\pi_0\Tot_s(\Pi^\bullet\bfC)$, we have to find a lift of its image in the pullback
\[
P=\Map(\partial\Delta^s,\Pi^s\bfC)
\times_{\Map(\partial\Delta^s,M^{s-1}\Pi^\bullet\bfC)}
\Map(\Delta^s,M^{s-1}\Pi^\bullet\bfC)
\]
as illustrated in the following diagram.
\begin{center}
  \mbox{ 
    \xymatrix{
\partial\Delta^s\ar[r]\ar[d]&\Pi^s\bfC\ar[d]\\
\Delta^s\ar[r]\ar@{-->}[ur]&M^{s-1}\Pi^\bullet\bfC
    } 
  }
\end{center}
Clearly, the element in~$\pi_{s-1}(\Pi^s\bfC)$ represented by~$\partial\Delta^s\to\Pi^s\bfC$ is an obstruction to the existence of such a lift. In fact, this obstruction lives on the~$E_1$ term, more precisely in~$E_1^{s,s-1}$. It turns out that we can pass to the~$E_2$ term to obtain a well-defined obstruction in~$E_2^{s,s-1}$ to the liftability of the restriction in~$\pi_0\Tot_{s-1}(\Pi^\bullet\bfC)$ of our element in~$\pi_0\Tot_s(\Pi^\bullet\bfC)$ to~$\pi_0\Tot_{s+1}(\Pi^\bullet\bfC)$.

\begin{remark}\label{rem:pi1}
A little extra care has to be taken when there are fundamental groups involved. It is common to assume that Whitehead products vanish in the spaces involved in~$\Pi^\bullet\bfC$, or at least that the fundamental groupoids act trivially on all fundamental groups. Both hypotheses are clearly satisfied in the linear case, when the mapping spaces are simplicial abelian groups, or in the stable case, when the mapping spaces are infinite loop spaces.
\end{remark}

The following result summarizes the situation.

\begin{theorem}\label{thm:ot}
Let~$\bfC$ be a simplicial category that is locally Kan and satisfies a hypothesis on the fundamental groups as in Remark~\ref{rem:pi1} above. Then any element in the center
\[
\rmZ(\Ho\bfC)=\Tot_1(\Pi^\bullet\bfC)
\]
of the homotopy category can be lifted to an element in the homotopy coherent center
\[
\calZ(\bfC)=\Tot(\Pi^\bullet\bfC)
\] 
if and only of if the obstruction classes in
\[
E_2^{s,s-1}=\pi^s\pi_{s-1}\Pi^\bullet\bfC
\]
vanish for all~$s\geqslant2$.
\end{theorem}

We note that the obstruction classes certainly vanish if the obstruction groups are all trivial. This happens, of course, in the case when~$\bfC$ is homotopically discrete.


\section{The category of finite groups}\label{sec:groups}

In this section we present a detailed discussion of the various notions of centers for the category of (discrete) groups. Of course, a size limitation needs to be chosen, and we can assume that all groups under consideration will be finite. Let~$\bfB_0$ denote a skeleton of the category of all such groups. This is the underlying category of a simplicial category~$\bfB$, where the mapping space~$\bfB(G,H)$ is the nerve of the groupoid of self-maps~$G\to H$. The objects in this groupoid are the homomorphisms~$G\to H$, and the morphisms between two homomorphisms are the elements~$h$ in~$H$ that conjugate one into the other. It is customary to denote the conjugacy classes of homomorphisms by
\[
\Rep(G,H)=\pi_0\bfB(G,H).
\]
The automorphism group of~$\alpha\colon G \to H$ is the centralizer
\begin{equation}\label{eq:Zalpha}
\rmC(\alpha)=\pi_1(\bfB(G,H),\alpha)
\end{equation}
of~$\alpha$. We remark that the centralizers need not be abelian. For example, the centralizer of the constant homomorphism~$G\to G$ is the entire group~$G$, whereas the center of~$G$ reappears as the centralizer of the identity~$G\to G$. These observations will be useful when we will determine the homotopy coherent center of~$\bfB$. Before we do so, let us first describe the centers of the ordinary categories~$\bfB_0$ and~$\Ho\bfB$. 

The homotopy category of groups is sometimes called the category of bands in accordance with its use in the non-abelian cohomology and theory of gerbes, see~\cite[IV.1]{Giraud}, where the notation~$\Hex(G,H)$ is used instead of our~$\Rep(G,H)$.

\begin{proposition}\label{prop:ordinary_groups}
The centers of the categories~$\bfB_0$ and~$\Ho\bfB$ are isomorphic to the abelian monoid~$\{0,1\}$ under multiplication.
\end{proposition}

\begin{proof}
This is straightforward for the category~$\bfB_0$ of groups and homomorphisms. An element in the center thereof is a family~$\Phi=(\Phi_G\colon G\to G)$ of homomorphisms that are natural in~$G$. We can evaluate~$\Phi$ on the full subcategory of cyclic groups, and since~$\Hom(\bbZ/k,\bbZ/k)\cong\bbZ/k$, we see that it is determined by a profinite integer~$n$ in~$\widehat\bbZ$: we must have~$\Phi_G(g)=g^n$ for all groups~$G$ and all its elements~$g$. But, if~$n$ is not~$0$ or~$1$, then there are clearly groups for which that map is not a homomorphism. In fact, we can take symmetric or alternating groups, as we will see below.

Let us move on to the center of the homotopy category~$\Ho\bfB$. Again, the homomorphisms~$g\mapsto g^0$ and~$g\mapsto g^1$ are in the center, and they still represent different elements, since they are not conjugate. In the homotopy category, if~$[\Phi]=([\Phi_G]\colon G\to G)$ is an element in the center, testing against the cyclic groups only shows that there is a profinite integer~$n$ such that~$\Phi_G(g)$ is conjugate to~$g^n$ for each group~$G$ and each of its elements~$g$. We will argue that no such family of homomorphisms~$\Phi_G$ exists unless~$n$ is~$0$ or~$1$.

Let us call an endomorphisms~$\alpha\colon G\to G$ {\it of conjugacy type~$n$} if~$\alpha(g)$ is conjugate to~$g^n$ for all~$g$ in~$G$. We need to show that for all~$n$ different from~$0$ and~$1$ there exists at least one group~$G$ that does not admit an endomorphism of conjugacy type~$n$.

If~$|n|\geqslant2$, then we choose~$m\geqslant\max\{n,5\}$ and consider the subgroup~$G$ of the symmetric group~$S_m$ generated by the elements of order~$n$. Since the set of generators is invariant under conjugation, this subgroups is normal, and it follows that~$G=A_m$~(the  subgroup of alternating permutations) or~$G=S_m$. An endomorphisms~$\alpha\colon G\to G$ of conjugacy type~$n$ would have to be trivial because it vanishes on the generators. But then~$g^n$ would be trivial for all~$g$ in~$A_m$, which is absurd.

If~$n=-1$, then we first note that an endomorphisms~$\alpha\colon G\to G$ of conjugacy type~$-1$ is automatically injective. Hence, if~$G$ is finite, then it is an automorphism. Therefore we choose a nontrivial finite group~$G$ of odd order such that its outermorphism group is trivial.~(Such groups exist, see~\cite{Horosevskii},~\cite{Dark} or~\cite{Heineken} for examples that also have trivial centers.) If~$\alpha\colon G\to G$ were an endomorphisms of conjugacy type~$-1$, then this would be an inner automorphism by the assumption on~$g$. Then~$\id\colon G\to G$, which represents the same homotopy class, would also be an endomorphisms  of conjugacy type~$-1$. In other words, every element~$g$ would be conjugate to its inverse~$g^{-1}$, a contradiction since the order of~$G$ is odd.
\end{proof}

We are now ready to apply the obstruction theory and spectral sequence to determine the homotopy coherent center of the category of groups.

\begin{proposition}
The maps
\[
\rmZ(\bfB_0)\longrightarrow\pi_0\calZ(\bfB)\longrightarrow\rmZ(\Ho\bfB)
\]
are both isomorphisms, and the component of the identity in~$\calZ(\bfB)$ is contractible.
\end{proposition}

\begin{proof}
Since the mapping spaces in~$\bfB$ are nerves of groupoids, they are~$1$-truncated: except for the fundamental groups, the higher homotopy groups vanish. This implies that the spectral sequence can be nontrivial only in the three spots~$E^{s,t}$ with~$t\leqslant 1$. This implies convergence (in fact, we have~$E_2=E_\infty$) and dramatically reduces the efforts needed to understand the target. However, these three entries are also the hardest to handle because {\it a priori} they need not be abelian groups.

Let us first deal with the one entry that has~$t=0$. We already know that there is an isomorphism~\hbox{$E^{0,0}_2\cong\rmZ(\Ho\bfB)$} by the general structure of the~$E_2$ term given in Theorem~\ref{thm:ss}. The center of the homotopy category has been determined in the preceding Proposition~\ref{prop:ordinary_groups}. That result also makes it clear that all elements in the center of the homotopy category lift to the homotopy coherent center; they even lift to the center of the underlying category~$\bfB_0$. We deduce that the obstructions vanish.

Let us now deal with the two entries that have~$t=1$ and that determine the kernel of the map~$\calZ(\bfB)\to\rmZ(\Ho\bfB)$ and the fundamental group of~$\calZ(\bfB)$ at the identity element. We have
\[
E^{0,1}_1=\prod_F\pi_1(\bfB(F,F),\id)\cong\prod_F\rmZ(F)
\]
by~\eqref{eq:Zalpha}, and we record that this is an abelian group. 

In order to determine~$E^{1,1}_1$, we start with 
\[
\pi_1\Pi^1\bfB\cong\prod_{G,H}\pi_1(\Map(\bfB(G,H),\bfB(G,H)),\id).
\]
Now we recall that~$\bfB(G,H)$ is the disjoint union of the classifying spaces for~$\rmC(\alpha)$,
where~$\alpha$ runs through a system of representatives of~$\Rep(G,H)$ in~$\Hom(G,H)$. Since we are considering loops based at the identity, we get
\[
\pi_1(\Map(\bfB(G,H),\bfB(G,H)),\id)
\cong\prod_{[\alpha]\in\Rep(G,H)}\rmZ\rmC(\alpha),
\]
the product of the centers of the centralizers. We note again that this is an abelian group. This leaves us with 
\[
\pi_1\Pi^1\bfB\cong\prod_{G,H,[\alpha]\in\Rep(G,H)}\rmZ\rmC(\alpha).
\]
The codegeneracy homomorphism is the evaluation at the identities. Consequently, we end up with
\begin{equation}\label{eq:cocycles}
E^{1,1}_1\cong\prod_{G,H,[\alpha]\not=\id}\rmZ\rmC(\alpha),
\end{equation}
the subgroup of normalized cochains.

The differential~$\delta\colon E^{0,1}_1\to E^{1,1}_1$ is the difference of the two coface homomorphisms. Therefore, up to an irrelevant sign, it is given on a family~$\Phi=(\Phi(F)\in\rmZ(F)\,|\,F)$ by
\begin{equation}\label{eq:diff1}
(\delta\Phi)(\alpha)=\Phi(H)-\alpha(\Phi(G))\in\rmZ\rmC(\alpha)
\end{equation}
in the factor of~$\alpha\colon G\to H$. It follows that the~$E^{0,1}_2$ entry consist of those families~$\Phi$ such that~$\Phi(H)=\alpha(\Phi(G))$ for all~$G$,~$H$, and~$\alpha\colon G\to H$. Taking~$\alpha$ to be constant, we see that~$\Phi(F)$ has to be trivial for all~$F$. This shows
\[
E^{0,1}_2=0.
\]
Since~$E^{s,t}_2=0$ for all other~$s,t$ such that~$t-s=1$, we deduce that~$\pi_1(\calZ(\bfB),\id)$ is trivial, so that the component is contractible. It remains to be shown that there are no more components than we already know, and these are indexed by~$E^{1,1}_2$.

The group~$E^{1,1}_2$ is isomorphic to the~$1$-cocyles in~\eqref{eq:cocycles}, those normalized~$1$-cochains on which the alternating sum of the coface maps vanishes. These coface maps sends a normalized~$1$-cochain~\hbox{$\Psi=(\Psi(\alpha)\in\rmZ\rmC(\alpha)\,|\,\alpha)$} to the families of
\[
\Psi(\gamma), \Psi({\gamma\beta}), \gamma(\Psi(\beta))\in\rmZ\rmC(\gamma)\leqslant N,
\]
respectively. These families are indexed by the composable pairs~$L\overset{\beta}{\to}M\overset{\gamma}{\to}N$ each time. Since the elements in~$\rmZ\rmC(\gamma)$ commute, we can write
\[
(\delta\Psi)(\gamma,\beta)=\Psi(\gamma)-\Psi(\gamma\beta)+\gamma(\Psi(\beta)),
\]
again up to an irrelevant sign. We claim that each element in the kernel of this~$\delta$ is already in the image of the previous~$\delta$ described in \eqref{eq:diff1}.

To prove this, let us be given a family~$\Psi=(\Psi(\alpha))$ such that
\begin{equation}\label{eq:cocycle}
\Psi(\gamma\beta)=\Psi(\gamma)+\gamma(\Psi(\beta)).
\end{equation} 
We can evaluate this family at the unique homomorphisms~$\alpha=\epsilon_F$ from the trivial group to~$F$, for each~$F$, to obtain the family~$\Phi(F)=\Psi(\epsilon_F)$ that is our candidate for an element~$\Phi$ to hit~$\Psi$. And indeed, equation~\eqref{eq:cocycle} for~$\alpha=\gamma$ and~$\beta=\epsilon_G$ gives
\[
\Psi(\epsilon_H)=\Psi(\alpha\epsilon_G)=\Psi(\alpha)+\alpha(\Psi(\epsilon_G)).
\]
Rearranging this yields the identity
\[
\Psi(\alpha)=\Psi(\epsilon_H)-\alpha(\Psi(\epsilon_G))=\Phi(H)-\alpha(\Phi(G))=(\delta\Phi)(\alpha),
\]
and this shows that~$\Psi$ is indeed in the image. We have proved the claim, so that we now know~$E^{1,1}_2=0$, and this finishes the proof. 
\end{proof}

It seems reasonable to expect that similar arguments will determine the homotopy coherent centers of related categories such as the category of groupoids, the category of~$1$-truncated spaces, etc. This will not be pursued further here. Instead, we will now turn our attentions towards a class of examples that indicates the wealth of obstructions and nontrivial differentials that one can expect in general.


\section{Groupoids}\label{sec:groupoids}

We will now discuss another class of examples of simplicial categories, namely simplicial groupoids. By definition, a simplicial groupoid~$\bfG$ is a simplicial category such that the categories~$\bfG_n$ of~$n$-simplices are groupoids. For example, simplicial groups, as discussed already in Section~\ref{sec:monoids}, are simplicial groupoids. Here, we are interested mainly in the case where there are many objects. Let us start by reviewing the discrete case first.

\subsection*{Fundamental groupoids}

If~$X$ is a space, its fundamental groupoid~$\Pi_1X$ is a (discrete) groupoid, and up to equivalence, every discrete groupoid has this form: Take~$X$ to be its classifying space.

\begin{proposition}\label{prop:center_PiX}
For every space~$X$, the center of the fundamental groupoid splits as a product 
\[
\rmZ(\Pi_1 X)\cong\prod_{[x]\in\pi_0X}\rmZ(\pi_1(X,x))
\]
of centers of its fundamental groups over a set of representatives of its path components.
\end{proposition}

\begin{proof}
This is certainly true for path connected spaces~$X$, since in that case the fundamental groupoid is equivalent to the fundamental group of any of its points. The general case follows from the compatibility of centers with disjoint unions.
\end{proof}

After this review of the discrete case, let us now move on to simplicial groupoids.


\subsection*{Loop groupoids}

While the fundamental groupoid~$\Pi_1 X$ of a space~$X$ is useful for many purposes, the passage to homotopy classes is a rather drastic simplification of the situation. It is preferable to work with the loop groupoid~$\bfG X$ of~$X$ that has been introduced by Dwyer and Kan in~\cite{Dwyer+Kan:groupoids}. It is a simplicial groupoid that can be thought of as a simplicially enriched refinement of the fundamental groupoid of the space~$X$. In fact, in the cited paper it is shown that the homotopy theory of simplicial groupoids is equivalent to the homotopy theory of spaces. More precisely, inverse equivalences are given by the pair of adjoint functors~\hbox{$\bfG\mapsto\rmB\bfG$}, the classifying space, and~$X\mapsto\bfG X$. Consequently, one may expect the centers of simplicial groupoids to be related to the classical homotopy theory of spaces, and as we will see in the examples below, this turns out to be true.  

The homotopy category of~$\bfG X$ is the fundamental groupoid~$\Pi_1 X$ of~$X$, so that the formula
\[
\Pi_1 X=\Ho(\bfG X),
\]
relates the two incarnations of the fundamental groupoid idea.

The center of the fundamental groupoid of~$X$ has been described in Proposition~\ref{prop:center_PiX}, and we can now ask for the homotopy coherent center~$\calZ(\bfG X)$ of the Dwyer-Kan loop groupoid of~$X$, and the relation between the two.


The homotopy coherent center in general has an~$E_2$ structure by Theorem~\ref{thm:E2}. {\it A fortiori}, it is an~$E_1=A_\infty$ monoid. For simplicial groupoids, it will turn out to be group-like by Corollary~\ref{cor:quasi_center_groupoid_abelian} below. Therefore, by Stasheff's recognition theorem~\cite{Stasheff:recognition} (as improved by May~\cite{May:recognition}), it must be equivalent to the loop space of some (classifying) space. In fact, using the~$E_2$ structure, that space will have its own delooping as well. The reader may wonder what these spaces are, and the following results will answer this question. 

\begin{proposition}\label{prop:center_formula}
For every simplicial groupoid~$\bfG$, we have
\[
\calZ(\bfG)\simeq\Omega(\Map(\rmB\bfG,\rmB\bfG),\id).
\]
\end{proposition}

\begin{proof}
Similarly to the preceding proposition, this follows from Corollary~\ref{cor:DWmodel} and the compatibility of centers with disjoint unions.
\end{proof}

The center of any groupoid is always an abelian group rather than just an abelian monoid. This is far from obvious in the homotopy coherent setting, but the following result affirms that it still holds true.

\begin{corollary}\label{cor:quasi_center_groupoid_abelian}
For every simplicial groupoid~$\bfG$, the abelian monoid~$\pi_0\calZ(\bfG)$ is an abelian group.
\end{corollary}

\begin{proof}
Proposition~\ref{prop:center_formula} allows us to identify~$\pi_0\calZ(\bfG)$ with a fundamental group of a space.
\end{proof}

If~$X$ is a Kan complex, then it is known that the mapping space~$\Map(X,X)$ already models the derived mapping space, and it is also a Kan complex. Since~$X\simeq\rmB\bfG X$, the proposition above has also the following corollary.

\begin{corollary}
For every Kan complex~$X$, we have
\[
\calZ(\bfG X)\simeq\Omega(\Map(X,X),\id).
\]
\end{corollary}

The description of spaces such as~$\Omega(\Map(X,X),\id)$ is a classical subject of (unstable) homotopy theory, and we will now discuss more specific classes of examples in order to demonstrate the complexity of the matter. 


\subsection*{Specific classes of spaces}

Given a space~$X$, we have the canonical homomorphism
\begin{equation}\label{eq:Gottlieb_map}
\pi_0\calZ(\bfG X)\longrightarrow\rmZ(\Pi_1X).
\end{equation}
This is the edge homomorphism of our obstruction theory spectral sequence in Section~\ref{sec:otss}, and we can discuss the problem whether or not this map~\eqref{eq:Gottlieb_map} is injective or surjective here. We will see that neither has to be the case.


\begin{example}
As the simplest case, consider a classifying space~$X=\rmB\Gamma=\rmK(\Gamma,1)$ for the discrete group~$\Gamma$. In this case, we should not expect any higher homotopy structure, and indeed, the space~$\Omega(\Map(X,X),\id)$ is homotopically discrete, with
\[
\calZ(\bfG X)\simeq\pi_0\calZ(\bfG X)\cong\rmZ(\Gamma),
\]
the center of the fundamental group.
\end{example}

In order to encounter higher homotopy structure, we may generalize this example in at least two ways: Replace~$X=\rmB\Gamma=\rmK(\Gamma,1)$ with an Eilenberg-Mac Lane space~$\rmK(A,n)$ for any abelian group~$A$ and any~$n\geqslant2$. Or, replace the aspherical circle~\hbox{$\rmB\bbZ=\rmK(\bbZ,1)=\rmS^1$} by a sphere~$\rmS^n$ of dimension~$n\geqslant2$. These will be our next two classes of examples.

\begin{example}\label{ex:EMspaces}
Let~$A$ be an abelian group, and let~$X=\rmK(A,n)$ be an Eilenberg-Mac Lane space of type~$(A,n)$ for any integer~$n\geqslant2$. Then the homotopy type of the space of self-maps is
\[
\Map(\rmK(A,n),\rmK(A,n))\simeq\Hom(A,A)\times\rmK(A,n),
\]
see~\cite{Thom} for the original argument and~\cite[25.2]{May} for the simplicial version. Since the space~$\Omega(\Map(X,X),\id)$ of loops based at the identity depends only on the component of the identity, we therefore obtain an equivalence
\[
\calZ(\bfG\rmK(A,n))\simeq\Omega\rmK(A,n)\simeq\rmK(A,n-1).
\]
This means that the ablelian group~$\pi_0\calZ(\bfG\rmK(A,n))$ is trivial. (In particular, it is  isomorphic to~$\rmZ(\Pi_1\rmK(A,n)))$ which is also trivial.) But, we see that the center can have arbitrary higher homotopy groups. This information will already be lost when one passes from the center~$\calZ(\bfG X)$ to its group~$\pi_0\calZ(\bfG X)$ of components, let alone~$\rmZ(\Pi_1 X)$.
\end{example}

\begin{example}\label{ex:spheres}
Let~$X=\rmS^n$ be an~$n$-dimensional sphere,~$n\geqslant2$. The homotopy groups of the homotopy coherent center~$\calZ(\bfG\rmS^n)$ can be computed from the fibration sequence
\[
\Omega^n\rmS^n\longrightarrow\Map(\rmS^n,\rmS^n)\longrightarrow\rmS^n.
\]
Let us first consider the stable case~$n\geqslant3$. In that case, in order to compute~$\pi_0$, the boundary operator~$0=\pi_2\rmS^n\to\pi_{n+1}\rmS^n$ must be zero and~$\pi_0\calZ(\bfG\rmS^n)$ is the stable~$1$-stem~$\bbZ/2$. The computation of the higher homotopy groups involves the boundary operator
\[
\pi_k\rmS^n\longrightarrow\pi_{k+(n-1)}\rmS^n,
\]
This is the Whitehead product with the identity~$\iota_n$ of~$\rmS^n$. In particular, the first potentially nonzero contribution is given by the Whitehead square of the identity, and that element figures prominently in the Hopf invariant one problem: the Whitehead square~$[\iota_n,\iota_n]$ is zero if and only if there is an element of Hopf invariant one in~$\pi_{2n+1}(S^{n+1})$, and this is very rarely the case (by~\cite{Adams}: only if~$n=3,7$ in our range). And, the divisibility of the class~$[\iota_n,\iota_n]$ is related to the (strong) Kervaire invariant one problem, see~\cite{BJM}. In the meta-stable case~$n=2$, in order to compute~$\pi_0\calZ(\bfG\rmS^2)$, the boundary operator~\hbox{$\bbZ\cong\pi_2(\rmS^2)\to\pi_3(\rmS^2)\cong\bbZ$} hits~$[\iota_2,\iota_2]$ which is divisible by~$2$ but not by~$4$, so that~$\pi_0\calZ(\bfG\rmS^2)\cong\bbZ/2$ also in this case. Indeed, for~$\rmS^2$, the entire homotopy type of the identity component of the mapping space is known, see~\cite{Hansen:1} and~\cite{Hansen:2}. The result is
\[
\Map(\rmS^2,\rmS^2)_{\id}\simeq\mathrm{SO}(3)\times\tilde{\Omega}^2_0\rmS^2,
\]
where~$\tilde{\Omega}^2_0\rmS^2$ is the universal cover of a component of the double loop space of the~$2$-sphere; it does not matter which component. The existence of such a homotopy equivalence also implies that the abelian group~$\pi_0\calZ(\bfG\rmS^2)$ has order~$2$, of course.
\end{example}

Example~\ref{ex:spheres} clearly shows that the map~\eqref{eq:Gottlieb_map} need not be injective, and surjectivity of~\eqref{eq:Gottlieb_map} can also fail in this context. In fact, the image has been studied from a different point of view in~\cite{Gottlieb}, and is commonly called the {\it Gottlieb subgroup} of the fundamental group. It can be as complicated as algebraically possible: If~$\Gamma$ is a discrete group, and~$G$ is a subgroup of its center, then there is a connected space~$X$ and an isomorphism~$\Gamma\cong\pi_1(X)$ such that this isomorphism sends~$G$ isomorphically to the Gottlieb subgroup of~$X$, see~\cite{Varadarajan}. 


\section*{Acknowledgment}

This research has been supported by the Danish National Research Foundation through the Centre for Symmetry and Deformation (DNRF92). I would like to thank Tyrone Crisp, Jesper Grodal, Ehud Meir, and Nathalie Wahl for their interest and comments, and I am especially grateful to Moritz Groth for many long discussions.



\vfill

\parbox{\linewidth}{%
Markus Szymik\\
Department of Mathematical Sciences\\
University of Copenhagen\\
Universitetsparken 5\\
2100 Copenhagen \O\\
Denmark\\
\phantom{ }\\
\href{mailto:szymik@math.ku.dk}{szymik@math.ku.dk}}


\end{document}